\documentclass[10pt]{amsart}
\usepackage{amsfonts, amssymb, amsmath}
\usepackage{mathtools}
\usepackage{tikz-cd}
\usepackage[all]{xy}

\usepackage[hidelinks]{hyperref}
\newcommand{\spa}{\medskip}
\usepackage[utf8]{inputenc}

\usepackage[left=25mm,right=25mm,top=38mm,bottom=34mm]{geometry}

\newcommand{\stack}[1]{\cite[\href{https://stacks.math.columbia.edu/tag/#1}{Tag #1}]{Stacks}}
\newcommand{\mr}[1]{{\mathrm{#1}}}
\newcommand{\F}{\mathbb{F}}
\newcommand{\Z}{\mathbb{Z}}
\newcommand{\iso}{\xrightarrow{\sim}}
\newcommand{\Zp}{\mathbb{Z}_p}
\newcommand{\Qp}{\mathbb{Q}_p}
\newcommand{\Fp}{\F_p}
\newcommand{\calC}{\mathcal{C}}
\newcommand{\calF}{\mathcal{F}}
\newcommand{\calO}{\mathcal{O}}
\newcommand{\calM}{\mathcal{M}}

\newcommand{\conv}{{\mr{conv}}}
\newcommand{\Spec}{\mr{Spec}}
\newcommand{\Spf}{\mr{Spf}}
\newcommand{\Hom}{\mr{Hom}}
\newcommand{\Ker}{\mr{Ker}}
\newcommand{\Spa}{\mr{Spa}}
\newcommand{\cotimes}{\widehat{\otimes}}
\begin{document}
	\newtheorem{theo}[subsubsection]{Theorem}
	\newtheorem*{theo*}{Theorem}
	\newtheorem{ques}[subsubsection]{Question}
	\newtheorem*{ques*}{Question}
	\newtheorem{conj}[subsubsection]{Conjecture}
	\newtheorem{prop}[subsubsection]{Proposition}
	\newtheorem{lemm}[subsubsection]{Lemma}
	\newtheorem*{lemm*}{Lemma}
	\newtheorem{coro}[subsubsection]{Corollary}
	\newtheorem*{coro*}{Corollary}

	\theoremstyle{definition}
	\newtheorem{defi}[subsubsection]{Definition}
	\newtheorem*{defi*}{Definition}
	\newtheorem{hypo}[subsubsection]{Hypothesis}
	\newtheorem{rema}[subsubsection]{Remark}
	\newtheorem{warn}[subsubsection]{Warning}
	\newtheorem{exam}[subsubsection]{Example}
	\newtheorem{nota}[subsubsection]{Notation}
		\newtheorem*{nota*}{Notation}
	\newtheorem{cons}[subsubsection]{Construction}

	\numberwithin{equation}{section}
\title{The convergent stack}
		\date{\today}
\author{Marco D'Addezio}
\address{Institut de Recherche Mathématique Avancée (IRMA)\\
Université de Strasbourg\\
7 rue René-Descartes\\
67000 Strasbourg, France}
\email{daddezio@unistra.fr}

\begin{abstract}
Inspired by Simpson's de Rham stack and Drinfeld's crystalline stack, we develop
a stacky approach to convergent cohomology and convergent isocrystals in positive
characteristic. To any scheme $X$ over $\mathbb{F}_p$ we attach a convergent stack
$X_{\mathrm{conv}}$. When $X$ is of finite type
over a perfect field, its finitely generated quasi-coherent $\mathcal{O}[\tfrac1p]$-modules
are equivalent to convergent isocrystals over $X$, compatibly with cohomology. When $X$
embeds into a smooth $p$-adic formal scheme, we describe $X_{\mathrm{conv}}$ explicitly as the
quotient of an open tube by a $p$-adic formal groupoid. For $f$-semiperfect schemes,
by contrast, the convergent stack is representable by a preperfectoid adic space over
$\mathbb{Q}_p$.
\end{abstract}
\maketitle

\tableofcontents

\section{Introduction}
\subsection{Main results}

The theory of crystalline cohomology provides a powerful tool for studying algebraic varieties and $p$-divisible groups in positive characteristic. To a smooth proper variety over a perfect field of characteristic $p>0$, it associates cohomology groups with coefficients in the ring of Witt vectors, which in many ways behave like the singular cohomology of a complex variety. However, for singular or non-proper varieties, crystalline cohomology is not well-behaved, and several variants have been introduced to address these problems. One such variant is \textit{convergent cohomology}, introduced by Berthelot \cite{Ber86} and reinterpreted by Ogus via the \textit{convergent site} \cite{Og84}. This yields a rational finitely generated cohomology theory for proper varieties.

\spa

In this article, we revisit Ogus' convergent site by allowing schemes that are not of finite type over a field, and we study the associated \textit{convergent stack} perspective, inspired by Simpson's de Rham stack \cite{Sim96}. 

\begin{defi}
	For a scheme $X$ over $\F_p$, we define the \textit{convergent stack of level $m$} of $X$ as the presheaf $X_\mr{conv}^{(m)}$ which assigns $$B\mapsto X(B/N_m(B)),$$ where $B$ is a $p$-complete $p$-torsion free ring and $N_m(B)$ is the ideal of elements $x\in B$ such that $x^{p^m}\in (p)$. The \textit{convergent stack} of $X$ is the colimit $$X_\mr{conv}\coloneqq\mr{colim}_{m\geq 0} X_\mr{conv}^{(m)}.$$ 
\end{defi}

Both $X_\mr{conv}^{(m)}$ and $X_\mr{conv}$ are stacks in sets for the $p$-completely étale topology (see Definition \ref{PEtale}). In the finite type setting, the cohomology of finitely generated quasi-coherent $\calO[\tfrac{1}{p}]$-modules over $X_\conv$ recovers the convergent cohomology of convergent isocrystals on $X$. In particular, we have the following comparison result.

\begin{theo}[Corollary \ref{GlobalconvDeRhamComparison:c}]\label{IntroGlobalconvDeRhamComparison:t}
Let $\mathfrak{Y}$ be a smooth quasi-compact $p$-adic formal scheme over $\Zp$. For every closed immersion $X\hookrightarrow \mathfrak{Y}\otimes_{\Zp}\Fp$ and every finitely generated quasi-coherent $\calO[\tfrac{1}{p}]$-module $\calM$ over $X_\conv$, there exists a canonical quasi-isomorphism
$$R\Gamma(X_{\conv}, \calM) \iso R\Gamma_\mr{dR}(]X[,\calM^\mr{ad}),$$ where $]X[$ is the open tube of $X$ in the generic fibre of $\mathfrak{Y}$ and $\calM^\mr{ad}$ is a flat connection over $]X[$ associated to $\calM$.
\end{theo}

We also prove relative versions of Theorem \ref{IntroGlobalconvDeRhamComparison:t} over Noetherian $p$-complete $p$-torsion free rings, generalising \cite[Thm. 0.6.6]{OgusCT}. For the proof, we adapt the arguments of \cite{BdJ11} in the context of convergent cohomology.

\spa
The convergent stack is closely related to the crystalline stack of
\cite{Dri22}, but it is in one respect a simpler object: it is a stack in
\textit{sets} rather than \textit{groupoids}. As with the crystalline stack,
	when $X$ admits an embedding into a smooth $p$-adic formal scheme $\mathfrak{Y}$, it can be described as a quotient of the $p$-adic open tube of $X$ in $\mathfrak{Y}$ by a certain ``convergent'' $p$-adic formal groupoid.
		\begin{theo}[Theorem \ref{Quotient:t}]\label{IntroQuotient:t}Let $\mathfrak{Y}$ be a smooth $p$-adic formal scheme over $\Zp$ and let $X\hookrightarrow \mathfrak{Y}\otimes_{\Zp} \Fp$ be a closed immersion. For every $m\geq 0$, there is a natural isomorphism $$X_\mr{conv}^{(m)} =\mathfrak{X}_0^{(m)}/\mathfrak{X}_1^{(m)},$$ where $\mathfrak{X}_1^{(m)}\rightrightarrows \mathfrak{X}_0^{(m)}$ is a $p$-adic formal groupoid (see Construction \ref{SimplicialTubes:c}).
\end{theo}
On the other hand, passing to \textit{$f$-semiperfect schemes}, the convergent stack is represented by  \textit{preperfectoid adic spaces} (see Definition \ref{FSemiPerfect:d} and \cite[Def. 3.7.1]{KL15}).
	\begin{theo}[Theorem \ref{ConvStackFunctor}]\label{IntroConvStackFunctor}
		The assignment $X \mapsto X_{\mr{conv}}$ induces a functor
		$$\left\{f\text{-}\mr{semiperfect}\ \mr{schemes}/\F_p\right\} \to \left\{\mr{Preperfectoid}\ \mr{spaces}/{\Qp}\right\}.$$
	\end{theo}
	This double nature offers a geometric
bridge between flat connections and $p$-adic periods, in the spirit of
Fontaine's theory. The relevant ring of periods, in our case, is the ring $\mathbb{A}_{\mathrm{conv}}^{(m)}$ which is a variant of $\mathbb{A}_\mr{cris}$ (see Construction \ref{Aconv:c}).

	\subsection{Further directions}
The crystalline stack of \cite{Dri22} was used in the proof of the Hecke
orbit conjecture to associate Dieudonné vector bundles to the universal
$p$-divisible groups of Shimura varieties \cite{DvH22}.  For
this application, it was important to assume that the deformation
spaces were formally smooth, so that crystalline cohomology had good properties
and one could work with equivariant vector bundles.  The convergent stack is designed to handle non-smooth
situations. We expect that it will help in understanding $p$-divisible
groups over more general deformation spaces.  A further goal is the
construction of a moduli space of convergent isocrystals. With the
stacky point of view, this reduces to constructing a moduli space of
vector bundles over the convergent stack.

\spa

The material of this article was originally part of \cite{EdgedCrystalline}. We defer the extension of the results to rigid cohomology and overconvergent isocrystals to [\textit{ibid.}]. The prismatic analogues will appear instead in \cite{EdgedPrismatic}.
\subsection*{Overview of the paper}

In \S\ref{sec:conv-revisited} we introduce the convergent stack in the relative setting
(Definition \ref{PEtale}) and develop its basic properties: change of
level and \'etale descent. We also define the category of convergent isocrystals in \S\ref{ConvergentIsocrystals:s}.
In \S\ref{sec:explicit} we give concrete descriptions of the
convergent stack.  We prove that when $X$ embeds into a smooth $p$-adic formal
scheme, the stack is presented as the quotient of a $p$-adic tube by a formal groupoid (Theorem \ref{Quotient:t}), and that for
$f$-semiperfect schemes it is representable by a
preperfectoid space (Theorem \ref{ConvStackFunctor}).  In
\S\ref{sec:computing} we compute convergent cohomology.  After
proving a convergent Poincar\'e lemma (Lemma
\ref{PoincareLemmaConvergent:l}) and a vanishing of higher limits
(Proposition \ref{MLZero:p}), we obtain comparison isomorphisms
between the cohomology of the convergent stack and the de Rham
cohomology of the tube (Theorems \ref{convDeRhamComparison1:t},
\ref{convDeRhamComparison2:t} and Corollary
\ref{GlobalconvDeRhamComparison:c}).
\subsection*{Acknowledgments}

This project began after a conversation with Kiran Kedlaya about the moduli space of $F$-isocrystals, following his talk on crystalline companions (see \cite{KedlayaCompanionsII}) at the conference ``$p$-adic cohomology and arithmetic geometry" in Sendai (2022). I would like to thank him for the stimulating discussion and the organisers for creating such a nice environment. I also thank Bhargav Bhatt, Léo Navarro Chafloque, and Peter Scholze for their valuable comments on a first draft.

\spa

The author was funded by the Deutsche Forschungsgemeinschaft (project ID: 461915680), the Marie Skłodowska-Curie Actions (project ID: 101068237), and the Agence Nationale de la Recherche (project ID: ANR-25-CE40-7869-01).

\section{Convergent cohomology revisited}
\label{sec:conv-revisited}
This section sets up the foundational framework. For a $p$-complete $p$-torsion free base ring $A$ and a scheme $X$ over $A/p$, we define the convergent
stack of $X$ over $A$ of level $m$ for every $m\geq 0$. In this way, we drop the
finite type hypothesis of Ogus' convergent site \cite{OgusCT}. We then study the
behaviour of the convergent stack under change of level (Lemma \ref{LimitModules:l}), prove étale
descent in $X$ (Proposition \ref{EtaleDescent:p}), introduce the coefficient
objects, and record a stacky form of Dwork's trick
(Lemma \ref{DworkTrick:l}).
\subsection{First definitions}
The convergent stack is built out of the operation $B\mapsto B_{[m]}$, which
kills the elements whose $p^m$-th power is divisible by $p$. This is the
algebraic shadow of passing to a $p$-adic neighbourhood of radius $p^{-1/p^m}$,
and letting $m$ grow recovers the full convergent geometry.
\begin{defi} \label{PEtale} For a $p$-complete ring $A$, we write $N_m(A)\subseteq A$ for the ideal of elements $x\in A$ such that $x^{p^m}\in (p)$. Equivalently, $N_m(A)$ is the kernel of the composition $$A\to A/p\xrightarrow{F^m}A/p.$$ We denote by $A_{[m]}\subseteq A/p$ the quotient $A/N_m(A)$. For an ideal $I\subseteq A$ containing $p$, we write $I^{(m)}$ for the ideal generated by $p$ and the elements $x^{p^m}$ with $x\in I$. A morphism of $p$-complete rings $A\to B$ is \textit{$p$-completely étale} if $A/p\to B/p$ is étale and $\mr{Tor}^1_A(B,A/p)=0$. 
\end{defi}

\begin{lemm}\label{RadicalPCompletelyEtale:l}
	If $A\to B$ is a morphism of $p$-complete rings with $A/p\to B/p$ étale, then $$\Ker(A/p\xrightarrow{F^m} A/p)\otimes_{A/p} B/p = \Ker(B/p\xrightarrow{F^m} B/p)$$ for every $m\geq 0$. In particular, $B_{[m]}=A_{[m]}\otimes_A B$ for every $m\geq 0$.
\end{lemm}
\begin{proof}
	By \stack{0F6W}, we have the following cocartesian diagram \[
\begin{tikzcd}
    A/p \arrow[r, "F^m"] \arrow[d] &
    A/p \arrow[d] \\
    B/p \arrow[r, "F^m"] &
    B/p
\end{tikzcd}
\] for every $m\geq 0$. Since $A/p\to B/p$ is flat, we deduce that $$\Ker(A/p\xrightarrow{F^m} A/p)\otimes_{A/p} B/p = \Ker(B/p\xrightarrow{F^m} B/p).$$ This implies that the right hand side is generated as a $B$-module by the image of $\Ker(A/p\xrightarrow{F^m} A/p)$, thus $N_m(A)B=N_m(B)$, as we wanted. 
\end{proof}

\begin{defi} For a $p$-complete ring $A$, we write $\mr{Alg}^{+}_{A}$ for the subcategory of $p$-complete $p$-torsion free $A$-algebras. We denote by $\otimes^{+}_A$ the tensor product in $\mr{Alg}^{+}_{A}$, i.e. the $p$-adic completion of the usual tensor product followed by the quotient by the $p$-torsion. A \textit{$p$-completely étale covering} of $(\mr{Alg}^{+}_{A})^\mr{op}$ is a family of morphisms $(B_i\to B)_{i\in I}$ such that each $B\to B_i$ is $p$-completely étale and $\coprod_i \Spec(B_{i}/p) \to \Spec(B/p)$ is surjective. We denote by $\mr{Stk}_A^{+}$ the category of stacks in sets over $(\mr{Alg}^{+}_{A})^\mr{op}$ endowed with the $p$-completely étale topology. We denote by $\mr{FormSch}^{+}_A$ the category of $p$-adic formal schemes over $A$ that are flat over $\Zp$.
\end{defi}

\begin{defi}
	Let ${X}$ be a scheme over $A/p$. For $m\in \Z_{\geq 0}$, an \textit{enlargement of level $m$} of ${X}$ over $A$ is the datum of an $A$-algebra ${B}\in \mr{Alg}^{+}_A$ endowed with a morphism $\Spec({B}_{[m]})\to{X}$. A morphism of enlargements $B\to B'$ is a morphism of rings $B'\to B$  which induces a commutative diagram \[
\begin{tikzcd}
    \mr{Spec}(B_{[m]}) \arrow[r] \arrow[d] &
    X  \\
    \mr{Spec}(B_{[m]}'). \arrow[ur] &
\end{tikzcd}
\]  We write $(X/A)_{\conv}^{(m)}$ for the category of enlargements of level $m$ over $A$. This category is naturally fibred over $(\mr{Alg}^{+}_A)^\mr{op}$ via the functor sending an enlargement to the underlying $A$-algebra.
\end{defi}
\begin{lemm}
The fibred category $(X/A)_{\conv}^{(m)}\to (\mr{Alg}^{+}_A)^\mr{op}$ is a stack in sets with respect to the $p$-completely étale topology.
\end{lemm}
\begin{proof}
	It is enough to prove that $B\mapsto X(B_{[m]})$ is a sheaf of sets for the $p$-completely étale topology. By Lemma \ref{RadicalPCompletelyEtale:l}, for every $p$-completely étale covering $(B_i\to B)_{i\in I}$ in $(\mr{Alg}_A^{+})^\mr{op}$ and every $m\geq 0$, the induced morphisms $((B_i)_{[m]}\to B_{[m]})_{i\in I}$ form an étale covering in $(\mr{Alg}_{(A/p)})^\mr{op}$. Thus the result follows from the fact that $X$ is a sheaf for the étale topology.
\end{proof}
\begin{defi}The stack $(X/A)_{\conv}^{(m)}$ is the \textit{convergent stack of level $m$} of $X$ over $A$. For $B\in (\mr{Alg}^{+}_A)^\mr{op}$, we write $(X/A)_{\conv}^{(m)}(B)$ for the fibre of $(X/A)_{\conv}^{(m)}$ over $B$. By construction, we have \begin{equation*} (X/A)_{\conv}^{(m)}(B)=X(B_{[m]}).\end{equation*}
	The stack $(X/A)_{\conv}^{(m)}$ is endowed with the topology induced by the $p$-completely étale topology on $(\mr{Alg}^{+}_A)^\mr{op}$. We say that the induced site is the \textit{convergent site of level $m$} of $X$ over $A$. 
We also write $\calO_\mr{conv}^+$ for the sheaf of $A$-algebras over $(X/A)_{\conv}^{(m)}$
	defined by $B\mapsto B$ and $\calO_\mr{conv}$ for the sheaf of $A[\tfrac{1}{p}]$-algebras defined by $B\mapsto B[\tfrac{1}{p}]$. 
\end{defi}

\subsection{Change of level}

For $m\leq m'$, the natural ring homomorphism $B_{[m]}\to B_{[m']}$ induces a $1$-morphism of stacks $$j\colon(X/A)_{\conv}^{(m)}\to (X/A)_{\conv}^{(m')}$$ and for $B\in \mr{Alg}_A^{+}$, we have $$j(B)\colon (X/A)_{\conv}^{(m)}(B)\hookrightarrow (X/A)_{\conv}^{(m')}(B).$$ 

\begin{defi}
	We denote by $(X/A)_{\conv}$ the $2$-colimit of the stacks $(X/A)_{\conv}^{(m)}$ along the morphisms $j\colon(X/A)_{\conv}^{(m)}\to (X/A)_{\conv}^{(m')}$ for $m\leq m'$. We write $j\colon(X/A)_{\conv}^{(m)}\to (X/A)_{\conv}$ for the induced morphism of stacks.
\end{defi}
By \stack{0796}, the functor $$j^*\colon \mr{Mod}((X/A)_{\conv},{\calO^+_\conv})\to \mr{Mod}((X/A)_{\conv}^{(m)},{\calO^+_\conv}),$$ admits a left adjoint $j_!$. For $\calM\in \mr{Mod}((X/A)_{\conv}^{(m)},{\calO^+_\conv})$, the sheaf $j_!\calM$ is the sheafification of the  presheaf $j_!^p$ sending $B\in  (X/A)_{\conv}$ to $\calM(B)$ if $B\in (X/A)_{\conv}^{(m)}$ and to $0$ otherwise. In particular, $j_!$ is exact and the counit $$j_!j^*\calM\hookrightarrow \calM$$ is injective.

\begin{lemm}\label{LimitModules:l}
	For $\calM\in \mr{Mod}((X/A)_{\conv},{\calO^+_\conv})$, if we write $\calM_m$ for the restriction of $\calM$ to $(X/A)_{\conv}^{(m)}$, then 
	$$R\Gamma((X/A)_{\conv}, \calM) = R\lim_m R\Gamma((X/A)_{\conv}^{(m)}, \calM_m).$$
\end{lemm}
\begin{proof}Since $\calO_\mr{conv}^+$ is flat over $\Z$, we may compute the cohomology groups in the category of $\calO_\mr{conv}^+$-modules.
	Thus it is enough to show the result when $\calM$ is injective. Since $$j_!\colon \mr{Mod}((X/A)^{(m)}_{\conv},{\calO^+_\conv}) \to \mr{Mod}((X/A)_{\conv},{\calO^+_\conv})$$ is exact, the modules $\calM_m$ are injective as well. Thus we have $R\Gamma((X/A)_{\conv}, \calM)=\Gamma((X/A)_{\conv}, \calM)$ and $R\Gamma((X/A)_{\conv}^{(m)}, \calM_m)=\Gamma((X/A)_{\conv}^{(m)}, \calM_m)$ for every $m\geq 0$. By definition, we have $$\Gamma((X/A)_{\conv}, \calM)=\lim_m \Gamma((X/A)_{\conv}^{(m)}, \calM_m),$$ thus it remains to show that $$R^1\mr{lim}_m\Gamma((X/A)_{\conv}^{(m)}, \calM_m)$$ vanishes. But $$\mr{Hom}(\calO^+_\mr{conv},\calM)\to \mr{Hom}(j_!\calO^+_\mr{conv},\calM) $$ is surjective because $\calM$ is an injective module and $j_!\calO_\mr{conv}^+\hookrightarrow \calO_\mr{conv}^+$ is injective.
\end{proof}

\subsection{Étale descent}

The construction $X\mapsto (X/A)_\conv^{(m)}$ is functorial: a morphism 
$f\colon Y\to X$ of schemes over $A/p$ defines a $1$-morphism of stacks 
$f_\conv\colon (Y/A)_\conv^{(m)}\to (X/A)_\conv^{(m)}$ sending 
$\xi\colon \Spec(B_{[m]})\to Y$ to $f\circ\xi$. We show that this 
construction satisfies étale descent in $X$.

\begin{lemm}\label{EtaleLifting:l}
For $B\in \mr{Alg}_A^{+}$ and $m\geq 0$, base change along 
$B\to B_{[m]}$ induces an equivalence
$$\{p\textrm{-completely étale }B\textrm{-algebras}\}
\xrightarrow{\sim}
\{\textrm{étale }B_{[m]}\textrm{-algebras}\},$$
under which $p$-completely étale coverings correspond to étale coverings.
\end{lemm}
\begin{proof}
Since $B$ is 
complete for the $N_m(B)$-adic topology, the pair $(B,N_m(B))$ is Henselian, and by \stack{09XI} every 
étale $B_{[m]}$-algebra lifts uniquely to an étale $B$-algebra. Its 
$p$-adic completion is a $p$-completely étale $B$-algebra. The compatibility with coverings is immediate from 
Lemma \ref{RadicalPCompletelyEtale:l}, which yields 
$B'_{[m]}=B_{[m]}\otimes_B B'$ for every $p$-completely étale morphism 
$B\to B'$.
\end{proof}

\begin{prop}\label{EtaleDescent:p}
Let $\{j_i\colon X_i\to X\}_{i\in I}$ be an étale covering of schemes 
over $A/p$ and, for $n\geq 0$, write 
$$X_n\coloneqq \coprod_{(i_0,\ldots,i_n)\in I^{n+1}} 
X_{i_0}\times_X\cdots\times_X X_{i_n}$$
for the associated Čech nerve. The natural $1$-morphism
$$\mr{colim}_{[n]\in\Delta^{\mr{op}}}(X_n/A)_{\conv}^{(m)}
\to (X/A)_{\conv}^{(m)}$$
is an isomorphism in $\mr{Stk}_A^{+}$. 
\end{prop}
\begin{proof}
Since the colimit of the simplicial diagram of 
stacks $(X_\bullet/A)_\conv^{(m)}$ is the sheafification of the colimit 
of the underlying simplicial diagram of presheaves of sets, it suffices 
to check that for every $B\in\mr{Alg}_A^{+}$ the natural map 
$$ 
\mr{colim}_{[n]\in\Delta^{\mr{op}}}X_n(B_{[m]})\to X(B_{[m]})$$
is injective and surjective up to a $p$-completely étale refinement 
of $B$.

\spa

For injectivity, suppose that $\xi_1\in X_{i_1}(B_{[m]})$ and 
$\xi_2\in X_{i_2}(B_{[m]})$ have the same image in $X(B_{[m]})$. Then 
the pair $(\xi_1,\xi_2)$ defines a $B_{[m]}$-point of 
$X_{i_1}\times_X X_{i_2}\subseteq X_1$ whose two faces are $\xi_1$ and 
$\xi_2$. Thus $\xi_1$ and $\xi_2$ are identified in the colimit.

\spa

For surjectivity, take $\xi\in X(B_{[m]})$. Pulling back 
$\{X_i\to X\}_{i\in I}$ along $\xi$ and refining by an affine Zariski 
covering produces an étale covering 
$$\{\Spec(C_k)\to \Spec(B_{[m]})\}_{k\in K}$$ with the property that for 
every $k\in K$ the morphism $\Spec(C_k)\to X$ factors through some 
$X_{i(k)}\to X$. By Lemma \ref{EtaleLifting:l}, this étale covering of 
$\Spec(B_{[m]})$ is the base change of a $p$-completely étale covering 
$\{B\to B_k\}_{k\in K}$ with $(B_k)_{[m]}=C_k$. The factorisation 
$\Spec((B_k)_{[m]})\to X_{i(k)}$ realises the restriction of $\xi$ to 
$(B_k)_{[m]}$ as the image of an element of $$X_{i(k)}((B_k)_{[m]})
\subseteq X_0((B_k)_{[m]}),$$ as required.
\end{proof}

\subsection{Convergent isocrystals and Dwork's trick}\label{ConvergentIsocrystals:s}
In this section, we introduce locally quasi-coherent sheaves and convergent isocrystals, and record Dwork's trick in our setting (Lemma \ref{DworkTrick:l}).
\begin{defi}
	 A sheaf of  $\calO_\conv$-modules $\calM$ over $(X/A)_{\conv}^{(m)}$ is \textit{locally quasi-coherent} if for every $p$-completely étale morphism $B_2\to B_1$ in $(X/A)_{\conv}^{(m)}$, the comparison morphism $$\calM(B_1){\otimes}_{B_1}B_2\to \calM(B_2)$$ is an isomorphism.
 A \textit{convergent isocrystal over ${X}$} is a sheaf of $\calO_\conv$-modules $\calM$ over $(X/A)_{\conv}$ such that for every morphism $B_2\to B_1$ in $(X/A)_{\conv}$, the comparison morphism $$\calM(B_1){\otimes}_{B_1}B_2\to \calM(B_2)$$ is an isomorphism.  We write $\mr{Isoc}_\mr{conv}({X}/A)$ for the category of convergent isocrystals over ${X}$ and we write $\mr{Isoc}_\mr{conv}^{\mr{fg}}({X}/A)$ for those isocrystals that are locally finitely generated.
\end{defi}

Dwork's trick can be reinterpreted as the fact that a Frobenius-invariant sheaf at finite level extends to the full convergent stack.

\begin{cons}\label{Dwork:c} Suppose that $A\in \mr{Alg}^+_{\Zp}$ is endowed with a Frobenius lift $\sigma\colon A\to A$. The absolute Frobenius morphism $F\colon X\to X$ induces a cocontinuous functor of sites $F\colon (X/A)_{\conv}^{(m)}\to (X/A)_{\conv}^{(m)}$ which sends $\Spec(B_{[m]})\to X$ to $$\Spec(B_{[m]})\xrightarrow{F} \Spec(B_{[m]})\to X.$$ The decomposition of $F\colon B_{[m]}\to B_{[m]}$ as the composition of the natural quotient $\pi\colon B_{[m]}\to B_{[m+1]}$ followed by $\varphi\colon B_{[m+1]}\to B_{[m]}$ induces a factorisation of $F\colon (X/A)_{\conv}^{(m)}\to (X/A)_{\conv}^{(m)}$ as the composition of cocontinuous functors $j\colon(X/A)_{\conv}^{(m)}\to (X/A)_{\conv}^{(m+1)}$ and $\varphi\colon(X/A)_{\conv}^{(m+1)}\to (X/A)_{\conv}^{(m)}$.
\end{cons}
\begin{lemm}[Dwork's trick]\label{DworkTrick:l}
	For $m\geq 0$, let $\calF$ be a sheaf over $(X/A)_{\conv}^{(m)}$ such that $F^*\calF\simeq \calF$. Then $\calF$ is the restriction of a sheaf over $(X/A)_{\conv}$.
\end{lemm}
\begin{proof}
	Using the decomposition $F=\varphi\circ j$, we have $$j^*\varphi^*\calF\simeq F^*\calF\simeq \calF.$$ Therefore, $\calF$ is the restriction of the sheaf $\varphi^*\calF$, defined over $(X/A)_{\conv}^{(m+1)}$. Iterating this process we get the desired result. 
\end{proof}

\section{Explicit descriptions}
\label{sec:explicit}
This section provides geometric models for the convergent stack in two regimes.
When $X$ embeds in a smooth $p$-adic formal scheme, the stack is a quotient of a
$p$-adic open tube by a formal groupoid (Theorem \ref{Quotient:t}), in analogy
with Drinfeld's presentation of the crystalline stack \cite{Dri22}. For
$f$-semiperfect schemes, the associated adic convergent stack is representable
by a preperfectoid space (Theorem \ref{ConvStackFunctor}). Along the way we
compare the convergent site with its topologically finitely generated
subcategory, which we will need in
\S\ref{sec:computing}.
\subsection{Tubes}
To make the convergent stack tangible we embed $X$ into a smooth $p$-adic formal scheme
and look at its $p$-adic closed tubes of radii $p^{-1/p^m}$.
\begin{cons}\label{Tubes:c}
	For a $p$-torsion free ring $B$ and an ideal $I\subseteq B$, we write $E^+_{B,m}(I)$ for the $p$-adic completion of $$B\left[\tfrac{(I,p)^{(m)}}{p}\right]$$ and $E_{B,m}(I)$ for $E^+_{B,m}(I)[\tfrac{1}{p}]$. We also denote by  $E_{B,\infty}(I)$ the limit  $\lim_m E_{B,m}(I)$ and by   $\Omega^\bullet_{E_{B,\infty}(I)/A}$ the limit of complexes $\lim_m \Omega^\bullet_{E^+_{B,m}(I)/A}[\tfrac{1}{p}]$. We often write $E_m$ and $E_\infty$ for  $E_{B,m}(I)$ and $E_{B,\infty}(I)$  when $B$ and $I$ are clear from the context.

\spa
  Let $A\in \mr{Alg}_{\Zp}^{+}$, let $\mathfrak{Y}\in \mr{FormSch}^{+}_A$, and let $X\hookrightarrow \mathfrak{Y}\otimes_{\Zp}\Fp$ be a closed subscheme over $A/p$. For any $m\geq 0$ we construct the tube $[X]^{(m)}\in \mr{FormSch}^{+}_A$ of level $m$ as follows. Working Zariski locally on $\mathfrak{Y}$, we may assume $\mathfrak{Y}=\Spf(\tilde{C})$ for some $\tilde{C}\in \mr{Alg}_A^{+}$ and that $X$ is defined by an ideal $I\subseteq \tilde{C}$ containing~$p$. We then define
  \[
    [X]^{(m)} \coloneqq\Spf\left(E^+_{\tilde{C},m}(I)\right).
  \]
\end{cons}
\begin{lemm}\label{Covering:l}
	In the situation of Construction~\ref{Tubes:c}, the tube $[X]^{(m)}$ represents the functor which assigns $B\in \mr{Alg}^{+}_A$ to the set of morphisms $\Spf(B)\to \mathfrak{Y}$ that carry $\Spec(B_{[m]})$ into $X$. In particular, there is a natural $1$-morphism $\pi\colon [X]^{(m)}\to (X/A)_\mr{conv}^{(m)}$ in $\mr{Stk}_A^{+}$.
\end{lemm}
\begin{proof}
	We may assume $\mathfrak{Y}=\Spf(\tilde{C})$ and $X$ defined by an ideal $I\subseteq \tilde{C}$ containing $p$. Since $B$ is $p$-complete and $p$-torsion free, $[X]^{(m)}(B)$ is the set of morphisms $\tilde{C}\to B$ sending $I^{(m)}$ into $pB$. This is equivalent to the condition that $\Spec(B_{[m]})\to \Spf(\tilde{C})$ factors through $X$.
\end{proof}

\begin{lemm}\label{Limits:l}
	If $X=\Spec(C)$ is affine, then $(X/A)_{\conv}^{(m)}$ has all finite nonempty limits.
\end{lemm}
\begin{proof}
	For $B_1,B_2\in (X/A)_{\conv}^{(m)}$ the fibre product is constructed as follows. Let $B$ be the tensor product $B_1 {\otimes}_A^{+} B_2$. Write $B/I$ for the pushout of the following diagram 
	\[
\begin{tikzcd}
    C\otimes_A C \arrow[r] \arrow[d] &
    C \arrow[d] \\
    (B_1)_{[m]}\otimes_A (B_2)_{[m]} \arrow[r] &
    B/I,
\end{tikzcd}\]
	where $C\otimes_A C\to C $ is the multiplication morphism. Then the natural morphism $E_{B,m}^+(I)\to{B}_{[m]}$ kills $I$ and induces by composition a morphism $C\to {B}_{[m]}$. The datum of $E_{B,m}^+(I)$ and $C\to (E_{B,m}^+(I))_{[m]}$ defines an enlargement of level $m$ over $A$ which is the fibre product of $B_1$ and $B_2$ in $(X/A)_{\conv}^{(m)}$.
\end{proof}

\begin{cons}
	\label{SimplicialTubes:c}
	Let $\mathfrak{Y}$ be a smooth formal scheme over $A$ and let $X\hookrightarrow \mathfrak{Y}\otimes_{\Zp} \Fp$ be a closed immersion. For every $m\geq 0$, we denote by $\mathfrak{X}_\bullet^{(m)}$ the simplicial formal scheme over $A$ obtained by taking the tube of level $m$ of the diagonal embedding of $X$ in $\mathfrak{Y}^{\bullet+1}$ (see Construction \ref{Tubes:c}).
	
\end{cons}

\begin{lemm}\label{FibreSquare:l} In the situation of Construction \ref{SimplicialTubes:c}, the following square of stacks
	\[
\begin{tikzcd}
    \mathfrak{X}_1^{(m)} \arrow[r,"d^1_1"] \arrow[d,"d^1_0"] &
    \mathfrak{X}_0^{(m)} \arrow[d,"\pi"] \\
    \mathfrak{X}_0^{(m)} \arrow[r,"\pi"] &
    (X/A)_\mr{conv}^{(m)},
\end{tikzcd}\] induced by the two projections $\mathfrak{Y}\times \mathfrak{Y} \to \mathfrak{Y}$, is strictly Cartesian. In addition, the double arrow \begin{equation}\label{Grp:eq}\mathfrak{X}_1^{(m)}\rightrightarrows \mathfrak{X}_0^{(m)},\end{equation}  is a groupoid in $\mr{FormSch}_A^{+}$.
\end{lemm}
\begin{proof} By Lemma \ref{Covering:l}, for $B\in \mr{Alg}_A^{+}$, the set $\mathfrak{X}_1^{(m)}(B)$ is the set of morphisms $\Spf(B)\to \mathfrak{Y}\times \mathfrak{Y}$ which send $\Spec(B_{[m]})$ into $X$. This corresponds to the set of pairs of morphisms $\Spf(B)\to \mathfrak{Y}$ which define the same morphism $\Spec(B_{[m]})\to X$.
In other words, the square of sets
	\[
\begin{tikzcd}
    \mathfrak{X}_1^{(m)}(B) \arrow[r] \arrow[d] &
    \mathfrak{X}_0^{(m)}(B) \arrow[d] \\
    \mathfrak{X}_0^{(m)}(B) \arrow[r] &
    (X/A)_\mr{conv}^{(m)}(B)
\end{tikzcd}\] is Cartesian. From this description, it is clear that the double arrow \eqref{Grp:eq} is a groupoid in $\mr{FormSch}_A^{+}$.
\end{proof}

\begin{theo}\label{Quotient:t}Let $\mathfrak{Y}\in \mr{FormSch}^{+}_A$ be a smooth formal scheme over $A$ and let $X\hookrightarrow \mathfrak{Y}\otimes_{\Zp} \Fp$ be a closed immersion. For every $m\geq 0$, there is a natural isomorphism $$(X/A)_\mr{conv}^{(m)} =\mathfrak{X}_0^{(m)}/\mathfrak{X}_1^{(m)},$$ where the quotient is taken in $\mr{Stk}_A^{+}$. In addition, we have $$(X/A)_\mr{conv}^{(m)} = \mr{colim}_{[n]\in \Delta^\mr{op}} \mathfrak{X}^{(m)}_n.$$ 
\end{theo}
\begin{proof} We want to prove that $(X/A)_\mr{conv}^{(m)}$ is the quotient stack of $\mathfrak{X}_0^{(m)}$ by $\mathfrak{X}_1^{(m)}$ via the $1$-morphism of stacks $\pi\colon \mathfrak{X}_0^{(m)}\to (X/A)_\mr{conv}^{(m)}$ of Lemma \ref{Covering:l}. For this, we may assume $\mathfrak{Y}=\Spf(\tilde{C})$ with $\tilde{C}$ the $p$-adic completion of a smooth $A$-algebra. By \stack{07K4}, the morphism $\mathfrak{Y}(B)\to \mathfrak{Y}(B_{[m]})$ is surjective for every $B\in \mr{Alg}_A^{+}$. Thanks to Lemma \ref{Covering:l}, this implies that $\mathfrak{X}_0^{(m)}(B)\to (X/A)_\mr{conv}^{(m)}(B)$ is surjective. By Lemma \ref{FibreSquare:l}, we deduce that $$\mathfrak{X}_0^{(m)}(B)/\mathfrak{X}_1^{(m)}(B) = (X/A)^{(m)}_\mr{conv}(B).$$ It remains to show that $\mathfrak{X}_\bullet^{(m)}$ is the nerve of the groupoid \eqref{Grp:eq}. This can be checked directly at the level of $B$-points  using Lemma \ref{Covering:l} as in Lemma \ref{FibreSquare:l}.
\end{proof}

\subsection{Topologically finitely generated algebras}
In this section, we use the results of the previous section to compare the convergent site with its subsite of topologically finitely generated algebras. We work over a Noetherian base ring $A\in \mr{Alg}_{\Zp}^{+}$ and a finitely generated affine $A$-scheme $X=\Spec(C)$ with $pC=0$. 

\begin{defi}
	We denote by $\mr{Alg}_A^{+,\mr{fg}}$ the full subcategory of $\mr{Alg}_A^{+}$ of topologically finitely generated $A$-algebras. We also denote by $(X/A)_{\conv}^{(m),\mr{fg}}$ the full subcategory of $(X/A)_{\conv}^{(m)}$ of those enlargements with underlying $A$-algebra in  $\mr{Alg}_A^{+,\mr{fg}}$.
\end{defi}

\begin{lemm}\label{ComparisonSmallBig:l}
	The inclusion $(X/A)_{\conv}^{(m),\mr{fg}}\hookrightarrow (X/A)_{\conv}^{(m)}$ preserves finite nonempty limits and is both continuous and cocontinuous. In particular, the induced morphism of topoi
\[
\iota\colon \mr{Ab}((X/A)_{\conv}^{(m),\mr{fg}}) \to \mr{Ab}((X/A)_{\conv}^{(m)})
\]
has the property that $\iota^{-1}$ admits an exact left adjoint $\iota_!$ satisfying $\iota^{-1}\iota_!=\iota^{-1}\iota_*=\mr{id}$.
	\end{lemm}
\begin{proof}
	The second part follows from the first part by \stack{04BH} and \stack{077I}.
\end{proof}

\begin{coro}\label{CohomologyComparison:c}
	For every sheaf of abelian groups $\calF$ over $(X/A)_{\conv}^{(m)}$, the natural morphism $$R\Gamma((X/A)_{\conv}^{(m)},\calF)\to R\Gamma((X/A)_{\conv}^{(m),\mr{fg}},\iota^{-1}\calF)$$ is a quasi-isomorphism.
\end{coro}
\begin{proof} By Lemma \ref{ComparisonSmallBig:l}, the functor $\iota^{-1}$ is exact and preserves injectives. Thus it suffices to prove the result when $\calF$ is injective. Choose a smooth $A$-algebra $B$ with a surjection to $C$ and write $E_m^+ \coloneqq E_{B,m}^+(I)$. By Theorem \ref{Quotient:t}, the global sections of $\calF$ and $\iota^{-1}\calF$ are both computed by the equaliser $\calF(E_m^+)\rightrightarrows \calF(E_m^+\times E_m^+)$, so the result follows.
\end{proof}
\begin{cons}For $B\in \mr{Alg}_A^{+}$, we denote by ${\Omega}^1_{B/A}$ the $p$-adic completion of the module of Kähler differentials of $B$ over $A$. The assignment $B\mapsto {\Omega}^1_{B/A}[\tfrac{1}{p}]$ defines a presheaf ${\Omega}^1_\mr{conv}$ on $(X/A)_{\conv}^{(m)}$. We write ${\Omega}^i_{\mr{conv}}$ for the $i$-th exterior power of ${\Omega}^1_\mr{conv}$.
\end{cons}

\begin{lemm}\label{DifferentialsTauEdgedCrystals:l}
	The $\calO_\conv$-modules ${\Omega}^i_\conv$ are locally quasi-coherent over $(X/A)_{\conv}^{(m),\mr{fg}}$.
\end{lemm}
\begin{proof}For a $p$-completely étale morphism $B_1\to B_2$, the induced morphisms $B_1/p^n\to B_2/p^n$ are étale for every $n\geq 1$, thus the morphism ${\Omega}^1_{B_1/A}\cotimes_{B_1}B_2\to {\Omega}^1_{B_2/A}$ is an isomorphism. If $B_1\in (X/A)_{\conv}^{(m),\mr{fg}}$, the module ${\Omega}^1_{B_1/A}$ is finitely presented, which implies that ${\Omega}^1_{B_1/A}\otimes_{B_1}B_2={\Omega}^1_{B_1/A}\cotimes_{B_1}B_2.$
\end{proof}
\subsection{Semiperfect schemes}\label{SemiPerfect:s}
We focus now on the convergent stack of \textit{$f$-semiperfect schemes} (see Definition \ref{FSemiPerfect:d}).
\begin{defi}\label{AlgN:d}
	Let $\mr{Alg}^{+,N}_{\Zp}$ denote the full subcategory of $\mr{Alg}^{+}_{\Zp}$ consisting of those $R$ that are integrally closed in $R[\tfrac1p]$, endowed with the $p$-completely étale topology. For such an $R$, the pair $(R[\tfrac1p],R)$ is a $p$-complete Huber pair and we write $\Spa(R[\tfrac1p],R)$ for the associated adic space. For a scheme $X$ over $\F_p$ and $m\geq 0$, the convergent stack $(X/\Zp)_{\mr{conv}}^{(m)}$ defines by restriction a stack over $\mr{Alg}^{+,N}_{\Zp}$, which we denote by $X_{\mr{conv}}^{\mr{ad},(m)}$. We also write $X_{\mr{conv}}^{\mr{ad}}$ for the restriction of $(X/\Zp)_{\mr{conv}}$ to $\mr{Alg}^{+,N}_{\Zp}$.
\end{defi}

\begin{cons}\label{Aconv:c}For a ring $R$ of characteristic $p$, we write $R^\flat$ for the tilt of $R$, i.e., the inverse limit perfection of $R$. Let $C$ be a semiperfect ring and write $J$ for $$\Ker(W(C^\flat)\to C^\flat\to C).$$ For every $m\geq 0$, we define  $$\mathbb{A}_\mr{conv}^{(m)}(C)\coloneqq E^+_{W(C^\flat),m}(J).$$ In other words, $\Spf(\mathbb{A}_\mr{conv}^{(m)}(C))$ is the tube of level $m$ of the closed immersion $\Spec(C)\hookrightarrow \Spf(W(C^\flat))$, as defined in Construction \ref{Tubes:c}. We write $\mathbb{B}_\mr{conv}^{(m)}(C)$ for $\mathbb{A}_\mr{conv}^{(m)}(C)[\tfrac{1}{p}]$ and $\mathbb{A}_\mr{conv}^{(m)}(C)^N$ for the $p$-adic completion of the integral closure of $\mathbb{A}_\mr{conv}^{(m)}(C)$ in $\mathbb{B}_\mr{conv}^{(m)}(C)$.
\end{cons}
We recall the following definitions from \cite[\S 4]{SW13}.

\begin{defi}\label{FSemiPerfect:d}
	An \textit{isogeny} of semiperfect rings is a surjective morphism $C\twoheadrightarrow D$ whose kernel is killed by a power of the Frobenius. A semiperfect ring $C$ is \textit{$f$-semiperfect} if it is isogenous to a semiperfect ring $D$ such that  $\Ker(D^\flat\to D)$ is finitely generated. A scheme $X$ over $\F_p$ is \textit{$f$-semiperfect} if it admits a Zariski covering by spectra of $f$-semiperfect rings.
\end{defi}

\begin{lemm}\label{FrobeniusIso:l}
	If $X$ is a semiperfect scheme, the absolute Frobenius induces an isomorphism of stacks $F\colon (X/\Zp)_{\conv}\iso (X/\Zp)_{\conv}$.
\end{lemm}
\begin{proof}We may assume that $X=\Spec(C)$ is affine. Thanks to Construction \ref{Dwork:c}, the absolute Frobenius on $(X/\Zp)^{(m)}_{\conv}$ factorises as the composition of the natural morphism $$j\colon (X/\Zp)^{(m)}_{\conv}\to (X/\Zp)^{(m+1)}_{\conv}$$ and a morphism $$\varphi\colon (X/\Zp)^{(m+1)}_{\conv}\to (X/\Zp)^{(m)}_{\conv}$$ induced by $\varphi\colon B_{[m+1]}\hookrightarrow B_{[m]}$, which sends $x\mapsto x^p$. Since $C$ is semiperfect, every morphism $C\to B_{[m]}$ factorises through $\varphi(B_{[m+1]})\subseteq B_{[m]}$. This shows that $$\varphi\colon (X/\Zp)^{(m+1)}_{\conv}\to (X/\Zp)^{(m)}_{\conv}$$ is an isomorphism of stacks. Passing to the colimit with respect to $m$, we deduce the desired result. 
\end{proof}
\begin{lemm}\label{TiltIsomorphism:l}
	For a $p$-complete ring $B$ and $m\geq 0$, the induced morphism $(B/p)^\flat\to (B_{[m]})^\flat$ is an isomorphism. In particular, for a perfect ring $R$, we have $$\Hom(W(R),B)=\Hom(R,B_{[m]}).$$
\end{lemm}
\begin{proof}
	Let $\pi\colon B/p \twoheadrightarrow B_{[m]}$ be the natural quotient map, and let $\varphi\colon B_{[m]}\to B/p$ be the map sending the class of $x\in B$ to $x^{p^m}\bmod p$. Both the composition $\varphi\circ\pi$ and $\pi\circ\varphi$ are the $m$-th absolute Frobenius.
	Applying the tilt functor we deduce that both  $\varphi^\flat\circ\pi^\flat$ and $\pi^\flat\circ\varphi^\flat$ are isomorphisms, thus both $\pi^\flat$ and $\varphi^\flat$ are isomorphisms. The second part of the statement follows from the first part and the universal property of Witt vectors.
\end{proof}

\begin{prop}\label{SemiPerfectRep:p}
	For an affine semiperfect scheme $X=\Spec(C)$, the $p$-adic formal scheme $\Spf(\mathbb{A}_\mr{conv}^{(m)}(C))$ represents the functor $(X/\Zp)_{\mr{conv}}^{(m)}$.
\end{prop}

\begin{proof}Let $\pi\colon \Spf(\mathbb{A}_\mr{conv}^{(m)}(C))\to (X/\Zp)_{\mr{conv}}^{(m)}$ be the $1$-morphism of stacks of Lemma \ref{Covering:l}. Thanks to Lemma \ref{TiltIsomorphism:l}, there is a natural bijection $$\Hom(W(C^\flat),B)=\Hom(C^\flat,B_{[m]}).$$ The morphisms $W(C^\flat)\to B$ which send $J$ to $N_m(B)$ then correspond to the morphisms $C^\flat \to B_{[m]}$ which factor through $C$. This shows that $\pi$ is an isomorphism.
\end{proof}



\begin{prop}\label{Perfectoid:p}
For a semiperfect ring $C$ such that $\Ker(C^\flat\to C)$ is finitely generated and $m\geq 0$, the pair $(\mathbb{B}_\mr{conv}^{(m)}(C),\mathbb{A}_\mr{conv}^{(m)}(C)^N)$ is preperfectoid and sheafy.
\end{prop}
\begin{proof}
 If $a_1,\dots,a_r$ are generators of $\Ker(C^\flat\to C)$, the ideal $J\coloneqq\Ker(W(C^\flat)\to C)$ is generated by $p$ together with $[a_i]$. It follows that $J^{(m)}$ is generated by $p$ and the elements $[a_i]^{p^m}$, so that $\mathbb{A}_\mr{conv}^{(m)}(C)$ is the $p$-adic completion of
$$W(C^\flat)\left[\tfrac{[a_1]^{p^m}}{p},\dots,\tfrac{[a_r]^{p^m}}{p}\right].$$

\spa

The pair $(W(C^\flat)[\tfrac1p],W(C^\flat))$ is a preperfectoid adic Banach ring: writing $(K,K^+)$ for the perfectoid affinoid pair obtained as the $p$-adic completion of $(\Qp[p^{1/p^\infty}],\Zp[p^{1/p^\infty}])$, the base change $W(C^\flat)\cotimes_{\Zp}K^+$ is integral perfectoid, since it is $p$-complete and the Frobenius $$(W(C^\flat)\cotimes_{\Zp}K^+)/p^{1/p}\to (W(C^\flat)\cotimes_{\Zp}K^+)/p$$ is an isomorphism.

\spa
The tube $\Spa(\mathbb{B}_\mr{conv}^{(m)}(C),\mathbb{A}_\mr{conv}^{(m)}(C)^N)$ is the rational localisation  $$\Spa(W(C^\flat)[\tfrac1p],W(C^\flat))\left(\frac{[a_1]^{p^m},\dots,[a_r]^{p^m}}{p}\right).$$ By \cite[Lem.~2.4.13]{KL15}, its ring of integral elements is the $p$-adic completion of the integral closure of $W(C^\flat)[\tfrac{[a_i]^{p^m}}{p}]$, which is exactly $\mathbb{A}_\mr{conv}^{(m)}(C)^N$. By \cite[Thm.~3.7.4]{KL15}, a rational localisation of a preperfectoid adic Banach ring is again preperfectoid and sheafy. This yields the desired result.
\end{proof}

	\begin{theo}\label{ConvStackFunctor}
		The assignment $X \mapsto X_{\mr{conv}}^\mr{ad}$ defines a functor
		$$\{f\text{-}\mr{semiperfect}\ \mr{schemes}/\F_p\} \to \{\mr{Preperfectoid}\ \mr{spaces}/{\Qp}\}.$$
	\end{theo}
	
	\begin{proof}
		Let $X=\Spec(C)$ be affine with $C$ $f$-semiperfect, choose an isogeny $C\to D$ with $\Ker(D^\flat\to D)$ finitely generated, and write $Y\coloneqq \Spec(D)$. By Lemma \ref{FrobeniusIso:l} the absolute Frobenius of $(X/\Zp)_{\conv}$ and $(Y/\Zp)_{\conv}$ are isomorphisms. By functoriality, this shows that $(X/\Zp)_{\conv}$ and $(Y/\Zp)_{\conv}$ are isomorphic stacks. We are thus reduced to proving the result for $Y$, which follows from Proposition \ref{SemiPerfectRep:p} and Proposition \ref{Perfectoid:p}. The general result is obtained by Zariski descent.
	\end{proof}

\section{Computing convergent cohomology}
\label{sec:computing}
The goal of this section is to compute the cohomology of finitely generated
coefficients on the convergent stack and to identify it with the de Rham
cohomology of the tube. We first prove a convergent Poincaré lemma and a
vanishing of higher limits across levels, then construct the functor from
convergent isocrystals to modules with integrable connection. Assembling these
via a Bhatt--de Jong cosimplicial argument yields the affine comparison
(Theorem \ref{convDeRhamComparison1:t}), which we extend to smooth embeddings
(Theorem \ref{convDeRhamComparison2:t}) and globalise (Corollary
\ref{GlobalconvDeRhamComparison:c}), recovering and extending Ogus'
comparison \cite[Thm.~0.6.6]{OgusCT}.
\subsection{The convergent Poincaré lemma}
The heart of any de Rham comparison is a Poincaré lemma: the assertion that the
relevant differentials have no cohomology beyond the constants. In the
convergent world this acquires the following form.
\begin{lemm}\label{PoincareLemmaConvergent:l}
	For $A\in \mr{Alg}^{+}_{\Zp}$ and $P$ a finitely generated polynomial $A$-algebra with augmentation ideal $I$, the complex $$A[\tfrac{1}{p}]\to \Omega^0_{E_\infty/A}\to \Omega^1_{E_\infty/A}\to \dots$$ is homotopy equivalent to zero, where $E_\infty\coloneqq E_{P,\infty}(I).$
\end{lemm}
\begin{proof}
	The result follows from the case of a polynomial algebra in one variable. In this case, it is enough to prove that $$\Omega^0_{E_\infty/A}\to \Omega^1_{E_\infty/A}$$ is surjective. Let $\sum_{i=0}^\infty a_i t^i dt$ be an element of $\Omega^1_{E_{m+1}/A}$. By construction, the coefficients $a_i$ satisfy the condition  $$\lim_{i\to \infty}v_p(a_i)+ip^{-m-1}=\infty.$$ The formal integral of $\sum_{i=0}^\infty a_i t^i dt$ is equal to $f\coloneqq \sum_{i=0}^\infty \frac{a_i}{i+1}t^{i+1}$. We want to prove that $f$ is an element of $E_m$, which means that $$\lim_{i\to \infty}v_p(a_i)-v_p(i+1)+(i+1)p^{-m}=\infty.$$ This follows from the observation that $$ip^{-m-1}-(i+1)p^{-m}+v_p(i+1)=(p^{-m-1}-p^{-m})i-p^{-m}+v_p(i+1)$$ is bounded from above.  
\end{proof}

\begin{nota}\label{Nota}
From now on and until the end of \S \ref{TheAffineDeRham:s} we use the following conventions. We fix a Noetherian ring $A\in \mr{Alg}_{\Zp}^{+}$, a topologically finitely generated $A$-algebra $B\in \mr{Alg}_A^{+,\mr{fg}}$, and an ideal $I\subseteq B$ such that $p\in I$. We write $C\coloneqq B/I$ and $X\coloneqq \Spec(C)$. We also write $E_m^+\coloneqq E_{B,m}^+(I),$ $E_m\coloneqq E_{B,m}(I)$, and $E_\infty\coloneqq E_{B,\infty}(I)$  as in Construction \ref{Tubes:c}.
\spa

We denote by $\mr{MIC}^\mr{fg}(B[\tfrac{1}{p}],I)$ the category of finitely generated $E_\infty$-modules $M$ endowed with a flat connection $$\nabla\colon M\to M\otimes_{E_\infty} \Omega^1_{E_\infty/A}.$$ For $(M,\nabla)\in \mr{MIC}^\mr{fg}(B[\tfrac{1}{p}],I)$, we write $M\otimes_{E_\infty}\Omega^\bullet_{E_\infty/A}$ for the associated de Rham complex.
\end{nota}
\begin{lemm}[cf. {\cite[Prop. 0.3.8]{OgusCT}}]\label{ConsequencePoincare:l}In the situation of Notation \ref{Nota}, let $P$ be a finitely generated polynomial $B$-algebra with augmentation ideal $J$, and let $K$ be the ideal $IP+J$ and   $E'_\infty\coloneqq E_{P,\infty}(K)$. For every $M\in \mr{MIC}^\mr{fg}(B[\tfrac{1}{p}],I)$, the natural morphism of complexes $$\iota\colon M\otimes_{E_\infty} \Omega^\bullet_{E_\infty/A}\to M\otimes_{E_\infty}\Omega^\bullet_{E'_\infty/A} $$ is a quasi-isomorphism.
\end{lemm}
\begin{proof}
We follow the filtration argument of \stack{07LD}, replacing the divided power Poincaré lemma with Lemma \ref{PoincareLemmaConvergent:l}. We have a split exact sequence 
\begin{equation}\label{ConvHodgeSES:eq}
0\longrightarrow \Omega^1_{E_\infty/A}\otimes_{E_\infty}E'_\infty\longrightarrow \Omega^1_{E'_\infty/A}\longrightarrow \Omega^1_{E'_\infty/B}\longrightarrow 0.
\end{equation} 

Let $F^*$ be the filtration on $\Omega^\bullet_{E'_\infty/A}$ defined by

$$F^i\left(\Omega^\bullet_{E'_\infty/A}\right)\coloneqq \Omega^{\geq i}_{E_\infty/A}\wedge\Omega^\bullet_{E'_\infty/A}.$$

The associated graded pieces are given by
$$\mr{gr}^i_F\left(\Omega^\bullet_{E'_\infty/A}\right)=\Omega^i_{E_\infty/A}\otimes_{E_\infty}\Omega^{\bullet}_{E'_\infty/B}[-i].$$
Tensoring with $M$ over $E_\infty$, we obtain a  decreasing filtration on $M\otimes_{E_\infty}\Omega^\bullet_{E'_\infty/A}$ with
$$\mr{gr}^i_F\bigl(M\otimes_{E_\infty}\Omega^\bullet_{E'_\infty/A}\bigr)\;=\;M\otimes_{E_\infty}\Omega^i_{E_\infty/A}\otimes_{E_\infty}\Omega^{\bullet}_{E'_\infty/B}[-i].$$

If we write $N_{i}\coloneqq M\otimes_{E_\infty}\Omega^i_{E_\infty/A}$ for $i\geq 0$, thanks to Lemma \ref{PoincareLemmaConvergent:l}, the morphism 
$$\mr{gr}^i_F(\iota)[i]\colon N_i\xrightarrow{\sim}N_i\otimes_{E_\infty}\Omega^\bullet_{E'_\infty/B}, $$
is a quasi-isomorphism for every $i\geq 0$. 
We deduce that $\iota$ is a quasi-isomorphism.
\end{proof}

\subsection{A vanishing of higher limits}

This section is an interlude on the vanishing of higher derived limits when $m$ varies. The proof is just an adaptation of \cite[Prop. 0.3.8]{OgusCT} to our setting.
\begin{lemm}[cf. {\cite[Lem. 0.3.4]{OgusCT}}]\label{ImageContainedVar:l}
In the situation of Notation \ref{Nota}, the image of $E^+_{m+k}\to E^+_m$ is contained in 
$B+p^k E^+_m$ for $k\geq 0$.
\end{lemm}
\begin{proof}
The ring $E^+_{m+k}$ is topologically generated over $B$ by the elements 
$x^{p^{m+k}}/p$ with $x\in I$. The identity
$$\frac{x^{p^{m+k}}}{p}=p^{\,p^k-1}\left(\frac{x^{p^m}}{p}\right)^{p^k}$$ shows that each such element lies in $p^k E_m^+$. 
\end{proof}

\begin{prop}[cf. {\cite[Prop. 0.3.8]{OgusCT}}]\label{MLZero:p}
Let $(M_m)_{m\geq 0}$ be a projective system of finitely 
generated $E_m^+$-modules (resp. $E_m$-modules) with 
transition maps $\psi_m$ inducing surjections
$$\tilde{\psi}_m\colon M_{m+1}\otimes_{E_{m+1}^+}E_m^+\twoheadrightarrow M_m.$$
Then, the higher derived limit $R^i\lim_m M_m$ vanishes for $i\geq 1$.
\end{prop}
\begin{proof}
We follow \cite[Prop. 0.3.8]{OgusCT} mutatis mutandis, with 
$I^{(m)}$ and $E_m^+$ playing the role of 
$J_n^k=I^k+pB$ and $B_n$ in \emph{loc. cit.}, and Lemma
\ref{ImageContainedVar:l} replacing Ogus' 
\cite[Lem. 0.3.4]{OgusCT}. The case $i\geq 2$ is trivial; for 
$i=1$ we show that every $x\in\prod_m M_m$ is a coboundary.

\spa

Set $F^k (M_m)\coloneqq \mr{Im}(M_{m+k}\to M_m)$ for $k,m\geq 0$. Arguing as 
in \emph{loc. cit.}, the surjectivity of $\tilde{\psi}_m$ together 
with Lemma \ref{ImageContainedVar:l} yields, for $k'\geq k$ and 
$m'\geq m$,
\begin{equation}\label{StabilisationVar:eq}
F^k M_m\subseteq \mr{Im}(F^{k'}M_{m'}\to M_m)+p^k M_m.
\end{equation}

Setting 
$$y_m\coloneqq x_m+x_{m+1}+\cdots+x_{2m-1}, \qquad 
x'_m\coloneqq x_{2m}+x_{2m+1},$$
where each $x_\ell$ is mapped to $M_m$ via the appropriate composition of 
$\psi$'s, one checks $\psi_m(y_{m+1})-y_m=x'_m-x_m$, so it suffices to 
show that $x'$ is a coboundary. Since $x'_m\in F^m M_m$, the containment
\eqref{StabilisationVar:eq} allows us to choose inductively 
$z_m\in F^m (M_m)$ and $x''_m\in p^m M_m$ such that
$$x'_m+z_m=\psi_m(z_{m+1})+x''_m,\qquad z_0=0.$$
The sum $s_m\coloneqq \sum_{\ell\geq 0}x''_{m+\ell}$ converges in $M_m$ 
because the image of $x''_{m+\ell}$ in $M_m$ lies in $p^{m+\ell}M_m$, and it 
satisfies $s_m-\psi_m(s_{m+1})=x''_m$. Hence $x''$, and consequently 
$x'$ and $x$, are coboundaries.

\spa

The case of $E_m$-modules is treated as in the last 
paragraph of \cite[Prop. 0.3.8]{OgusCT}, thanks to the density of $\psi_m$.
\end{proof}
\subsection{Isocrystals and connections}

In this section we construct a functor from the category of finitely generated convergent isocrystals to the category of modules with flat connection over the tube of $X$ in the affine setting. Recall that Notation \ref{Nota} is in force.

\begin{cons}\label{IsocrystalConnection:c}For $m\geq 1$, let $F_m^+$ be the square-zero split extension of $E_m^+$ by $(\Omega^1_{E_m^+/A})/p\textrm{-}\mr{tors}$ and let $F_m$ be $F_m^+[\tfrac{1}{p}]$. By construction, $(F_m^+)_{[m]}=(E_m^+)_{[m]}=C$, thus it admits a tautological enlargement structure of $X$ of level $m$. It is endowed with two morphisms $p_0,p_1\colon E_m^+\to F_m^+$ of rings, the first one being defined by $p_0(f)=(f,0)$ and the second one by $p_1(f)=(f,df)$. The morphism $s\colon F_m^+\to E_m^+$ sending $(f,\omega)$ to $f$ is a section of both $p_0$ and $p_1$. 

\spa
	
For $\calM\in\mr{Isoc}^\mr{fg}_\mr{conv}(X/A)$, if we write $M_m$ for $\calM(E_m^+)$, the isocrystal property supplies $F_m$-linear isomorphisms
$$ M_m\otimes_{E_m,p_0}F_m\underset{\sim}{\xrightarrow{\ c_0\ }} \calM(F_m)\underset{\sim}{\xleftarrow{\ c_1\ }}
M_m\otimes_{E_m,p_1}F_m,$$
and we set $c\coloneqq c_1^{-1}\circ c_0\colon p_0^*M_m\iso p_1^*M_m$. As
$s\circ p_0=s\circ p_1=\mr{id}_{E_m^+}$ we have $s^*c=\mr{id}_{M_m}$.
For $v\in M_m$, the element $\nabla(v)\coloneqq p_1^*v-c(p_0^*v)$ lies in
$M_m\otimes_{E_m,p_1}\Omega^1_{E_m/A}$, thus it defines an $A$-linear homomorphism
$$\nabla\colon M_m\to M_m\otimes_{E_m}\Omega^1_{E_m/A}.$$ Passing to the limit with respect to $m$ we get a homomorphism
$$\nabla\colon M\to M\otimes_{E_\infty}\Omega^1_{E_\infty/A}.$$
\end{cons}
\begin{lemm}\label{IsocrystalsToMIC:l} The assignment $\calM\mapsto(M,\nabla)$ of Construction~\ref{IsocrystalConnection:c}  defines a functor
$$\mr{Isoc}^\mr{fg}_\mr{conv}(X/A)\longrightarrow\mr{MIC}^\mr{fg}(B[\tfrac1p],I).$$
\end{lemm}
\begin{proof}
The Leibniz rule for $\nabla$ follows from the fact that $c$ is $F_m$-linear. To check that $\nabla$ is integrable, we define $$G_m^+\coloneqq E^+_m\oplus (\Omega^1_{E_m^+/A})/p\textrm{-}\mr{tors} \oplus (\Omega^1_{E_m^+/A})/p\textrm{-}\mr{tors} \oplus (\Omega^2_{E_m^+/A})/p\textrm{-}\mr{tors}$$ with multiplication given by $$(f, \omega _1, \omega _2, \eta ) \cdot (f', \omega _1', \omega _2', \eta ') = (ff', f\omega _1' + f'\omega _1, f\omega _2' + f'\omega _2, f\eta ' + f'\eta + \omega _1 \wedge \omega _2' + \omega _1' \wedge \omega _2).$$
In this case, $(G_m^+)_{[m]}=(E_m^+)_{[m]}=C$ for $m\geq 2$, thus $G_m^+$ admits a tautological enlargement structure of $X$ of level $m$. The integrability of $\nabla$ is then obtained as in \stack{07J6}. Passing to the limit with respect to $m$ we get the desired result.
\end{proof}

\subsection{The affine de Rham comparison}
\label{TheAffineDeRham:s}
We can now identify the cohomology of the convergent stack with the de Rham
cohomology of the tube in the affine case. The computation proceeds by
resolving the stack with the Čech nerve of a single enlargement and untangling
the resulting double complex, one direction governed by the Poincaré lemma and
the other by the simplicial structure. Even in this section, we adopt the conventions of Notation \ref{Nota}.

\begin{lemm}\label{conv-local-acyclicity:l}
	If $\calM$ is a locally quasi-coherent sheaf of $\calO_{\mr{conv}}$-modules over $(X/A)_{\conv}^{(m),\mr{fg}}$, then for every $B \in (X/A)_{\conv}^{(m),\mr{fg}}$ we have
	$$H^i(B,\calM) = 0$$
	for $i > 0$.
\end{lemm}
\begin{proof}By \stack{0912}, every $p$-completely étale cover of $B$ is an fpqc cover. The result then follows from the exactness of the Amitsur complex.
\end{proof}

\begin{coro}\label{ConvChaoticTopology:c}
	If we denote by $\calC^{(m)}$ the category $(X/A)_{\conv}^{(m),\mr{fg}}$ endowed with the chaotic topology and $\epsilon\colon \calC^{(m)}\to (X/A)_{\conv}^{(m),\mr{fg}}$ the functor of sites induced by the identity, then for every locally quasi-coherent sheaf of $\calO_{\mr{conv}}$-modules $\calM$ over $(X/A)_{\conv}^{(m),\mr{fg}}$, we have the following quasi-isomorphism 
	$$R\Gamma((X/A)_{\conv}^{(m)},\calM) \iso R\Gamma(\calC^{(m)},(\iota\circ\epsilon)^{-1}\calM).$$
\end{coro}
\begin{proof}
	By Lemma \ref{conv-local-acyclicity:l},  the higher direct images $R^i\epsilon_*\left((\iota\circ\epsilon)^{-1}\calM\right)$ vanish for $i>0$. The result then follows from the observation that $\epsilon_*\left((\iota\circ\epsilon)^{-1}\calM\right)=\iota^{-1}\calM$. See \stack{07JK} for more details.
\end{proof}

\begin{cons}\label{CechNerve:c} In the situation of Notation \ref{Nota}, suppose that $B$ is the $p$-adic completion of a smooth $A$-algebra.
	Let $(E_{m}^+(*), {C})$ be the \v{C}ech nerve of $E_{m}^+$ in $(X/A)_{\conv}^{(m),\mr{fg}}$ for $m\ge 0.$ 
	For a locally quasi-coherent sheaf of $\calO_\conv$-modules $\calM$ over $(X/A)_{\conv}^{(m),\mr{fg}}$, set $M_m(b) \coloneqq \calM(E_{m}^+(b))$ and write $M_m$ for $M_m(0)$. 
	We also write $\Omega_m^{\bullet} \coloneqq \Omega_{E_m/A}^{\bullet}$ and $\Omega_m^{\bullet}(n) \coloneqq \Omega_{E_m(n)/A}^{\bullet}$. 
	Finally, denote by $M(*)$ the limit $\lim_m M_m(*)$, by $M$ the limit $\lim_m M_m$, and set $E_\infty \coloneqq \lim_m E_m$. If $\calM$ is a finitely generated isocrystal, by Lemma \ref{IsocrystalsToMIC:l} the $E_\infty$-module $M$ is equipped with a flat connection $\nabla\colon M \to M\otimes_{E_\infty} \Omega^1_{E_\infty/A}$ which produces a de Rham complex  $$M\otimes_{E_\infty}\Omega^\bullet_{E_\infty/A}.$$
\end{cons}

\begin{prop}\label{InfFirstColumn:p}
	If $\calM$ is a locally quasi-coherent sheaf over $(X/A)_{\conv}^{(m),\mr{fg}}$, then there is a natural  quasi-iso\-morphism $$R\Gamma((X/A)_{\conv}^{(m),\mr{fg}},\calM)\iso (M_m(0)\to M_m(1)\to M_m(2)\to\dots ).$$
\end{prop}

\begin{proof}This is proved as in \stack{07JN}. By Corollary \ref{ConvChaoticTopology:c} we may work on the chaotic site $\calC^{(m)}$.  
Theorem \ref{Quotient:t} gives that $E_m^+$ is a weakly final object in $(X/A)^{(m),\mr{fg}}_\mr{conv}$, hence also in $\calC^{(m)}$.  For a locally quasi-coherent sheaf $\calM$, the restriction of $\calM$ to $\calC^{(m)}$ is a sheaf and its value on the Čech nerve is the cosimplicial module $M_m(*)$.  By \stack{07JM} the total cohomology of $\calM$ on $\calC^{(m)}$ is computed by the complex associated to $M_m(*)$. The proposition follows.
\end{proof}

\begin{lemm}\label{ConvAcyclicColumns:l}
	The cosimplicial $E_m(*)$-module
	$$M_m(0) \otimes_{E_m(0)} \Omega_m^i(0) \to M_m(1) \otimes_{E_m(1)} \Omega_m^i(1) \to \cdots$$
	is homotopic to zero for every $i \geq 1$.
\end{lemm}
\begin{proof} Arguing as in \stack{07L9}, the complex	$\Omega_{m}^{1}(*)$ is homotopic to zero as a $E_m(*)$-module. The result then follows from \stack{07KQ}.
\end{proof}

\begin{lemm}[Bhatt--de Jong]\label{BdJ:l} Let $\mathcal{A}$ be an abelian category and let 
	$$\xymatrix{
		K^{\bullet}(0)\ar[r]<1.5pt>\ar[r]<-1.5pt>& K^{\bullet}(1)\ar[r]<3pt>\ar[r]\ar[r]<-3pt>&\cdots
	}$$
	be a cosimplicial cochain complex of $\mathcal{A}$ such that for every $b\geq 0$ the complex	$K^{\bullet}(b)$ is concentrated in non-negative degrees. For $0\leq i \leq b$ write $\alpha_{i,b}:K^{\bullet}(0)\to  K^{\bullet}(b)$ for the morphism of complexes induced by the morphism $[0]\to [0,\dots,b]$ which sends $0$ to $i$. Suppose that the following conditions are satisfied.
	
	\begin{itemize}
		
		\item[\normalfont{(1)}] For every $0\leq i \leq b$, the morphism $\alpha_{i,b}$ is a quasi-isomorphism.
		\item[\normalfont{(2)}] For every $a>0$, the cochain complex associated to $K^{a}(*)$ is acyclic.
	\end{itemize}
	
	If  $K^{\bullet,*}$ is the double complex associated to $K^\bullet(*)$, then both $K^{0,*}$ and $K^{\bullet,0}$ are quasi-isomorphic to $\mr{Tot}(K^{\bullet,*})$.
\end{lemm}
\begin{proof}
	We first note that for every $0\leq i \leq b$ and $j\geq 0$, the isomorphisms $$\alpha_{i,b}\colon H^j(K^\bullet(0))\iso H^j(K^\bullet(b))$$ have as common inverse the morphism $$H^j(K^\bullet(b))\to H^j(K^\bullet(0))$$ induced by $[0,\dots,b]\to [0].$ In particular, they are independent of $i$. When taking the first spectral sequence associated to the double complex $K^{\bullet,*}$, the differentials of the first page are $$H^j(K^{0,*})\xrightarrow{0}H^j(K^{1,*})\xrightarrow{\sim}H^j(K^{2,*})\xrightarrow{0}H^j(K^{3,*})\to\dots$$
	This implies that $K^{0,*}$ is quasi-isomorphic to $\mr{Tot}(K^{\bullet,*})$. Looking at the second spectral sequence, Condition (2) implies that $K^{\bullet,0}$ is also quasi-isomorphic to $\mr{Tot}(K^{\bullet,*})$.
\end{proof}

\begin{theo}\label{convDeRhamComparison1:t}
	Suppose $A\in \mr{Alg}^{+}_{\Zp}$ Noetherian and let $P$ be a finitely generated polynomial $A$-algebra and let $I\subseteq P$ be an ideal containing $p$. If $X\coloneqq \Spec(P/I)$, then for every finitely generated convergent isocrystal $\calM$ of $X$ over $A$, there exists a canonical quasi-isomorphism
	$$R\Gamma((X/A)_{\conv}, \calM) \iso M\otimes_{E_\infty}\Omega^\bullet_{E_\infty/A},$$ where the right hand side is as in Construction \ref{CechNerve:c}.
\end{theo}
\begin{proof}
	For $m\geq 1$, by Corollary \ref{CohomologyComparison:c} and Corollary \ref{ConvChaoticTopology:c}, compute cohomology using the \v{C}ech nerve
	$${R}\Gamma((X/A)_{\conv}^{(m)}, \calM) \iso \left( M_m(0) \to M_m(1) \to M_m(2) \to \cdots \right).$$
	By combining Lemma \ref{LimitModules:l} and Proposition \ref{MLZero:p}, passing to the limit, we deduce that $${R}\Gamma((X/A)_{\conv}, \calM) \iso \left( M(0) \to M(1) \to M(2) \to \cdots \right).$$
	Consider the double complex $K^{\bullet,*}$ defined by
	$$K^{a,b} \coloneqq M \otimes_{E_\infty} \Omega_{E_\infty(b)/A}^{a}.$$
	By Lemma \ref{ConvAcyclicColumns:l}, the complex $K^{a,*}$ is acyclic for $a>0$. On the other hand, for fixed $b$, the complex $K^{\bullet,b}$ is the de Rham complex of $M$ over $E_\infty(b)$. By the isocrystal condition and Lemma \ref{ConsequencePoincare:l}, coface maps induce quasi-isomorphisms. Apply Lemma \ref{BdJ:l} to conclude that $\mr{Tot}(K^{\bullet,*})$ is quasi-isomorphic to both $K^{\bullet,0}$ and $K^{0,*}$.
\end{proof}

\subsection{The comparison for smooth formal schemes}
We extend the comparison from polynomial to smooth affine embeddings (Theorem
\ref{convDeRhamComparison2:t}). We then globalise the result in Corollary
\ref{GlobalconvDeRhamComparison:c}, extending \cite[Thm.~0.6.6]{OgusCT} in the relative setting.
\begin{cons}\label{SmoothEmbedding:c}
	For a Noetherian $A\in\mr{Alg}^{+}_{\Zp}$, let $A\to S\to C$ be ring maps with $A\to S$ smooth and $S\to C$ surjective with kernel $J$. Choose a finitely generated polynomial $A$-algebra $P$, a surjection $\pi\colon P\twoheadrightarrow S$ of $A$-algebras with kernel $K$. By the smoothness of $A\to S$ there exists a homomorphism
$$\sigma\colon S\longrightarrow P^\wedge_K$$
such that $\pi\circ \sigma=\mr{id}_{S}$, where $P^\wedge_K=\lim_n P/K^n$ denotes the $K$-adic completion of $P$ and $\pi\colon P^\wedge_K\twoheadrightarrow S$ is the surjection induced by $\pi$.
\end{cons}
\begin{lemm}\label{ConvergentAffineSmooth:l}In the situation of Construction \ref{SmoothEmbedding:c}, writing $I=\Ker(P\twoheadrightarrow C)$ and $E_m^+$ (resp. $F^+_m$) for $E^+_{P,m}(I)$ (resp. $E^+_{S,m}(J)$), the data $(\pi,\sigma)$ determines, simultaneously for every $m\geq 0$, $A$-algebra homomorphisms
$$a_m\colon E_m^+\to F_m^+,\qquad b_m\colon F_m^+\longrightarrow E_m^+,$$
compatible with the structural maps $C\to (E_m^+)_{[m]}$ and $C\to (F_m^+)_{[m]}$ and with the inclusions $E_{m+1}^+\hookrightarrow E_m^+$, $F_{m+1}^+\hookrightarrow F_m^+$, satisfying $a_m\circ b_m=\mr{id}_{F_m^+}$.
\end{lemm}

\begin{proof}
	To construct $a_m$, just note that $\pi$ induces homomorphisms $$P[\tfrac{I^{(m)}}{p}]\to S[\tfrac{J^{(m)}}{p}],$$ and $p$-adic completion yields morphisms $a_m\colon E_m^+\to F_m^+$ for $m\geq 0$.
For $b_m$, first note that since $I$ is finitely generated, the ring $E_m^+$ is $I$-adically complete for every $m\geq 0$ and, by construction, we have that $K\subseteq I$. We deduce that the natural morphism $P\to E_m^+$ extends canonically to a morphism $$\iota_m\colon P^\wedge_K\to E_m^+.$$ If $\beta_m\coloneqq \iota_m\circ \sigma$, then $\beta_m(J)\subseteq N_m(E_m^+)$, which implies $\beta_m\bigl(J^{(m)}\bigr)\subseteq pE_m^+$. The $p$-torsion-freeness of $E_m^+$ allows a unique extension of $\beta_m$ to an $A$-algebra homomorphism $S[\tfrac{J^{(m)}}{p}]\to E_m^+$. After $p$-adic completion, we obtain
$$b_m\colon F_m^+\to E_m^+.$$
The homomorphisms $a_m$ and $b_m$ satisfy the desired compatibilities by construction, and the identity $a_m\circ b_m=\mr{id}_{F_m^+}$ follows from the identity $\pi\circ \sigma=\mr{id}_{S}$.
\end{proof}

\begin{theo}\label{convDeRhamComparison2:t}
	For a Noetherian $A\in\mr{Alg}^{+}_{\Zp}$, let $A\to S\to C$ be ring maps with $A\to S$ smooth and $S\to C$ surjective with kernel $J$. For every finitely generated convergent isocrystal $\calM$ of $\Spec(C)$ over $A$, there exists a canonical quasi-isomorphism
	$$R\Gamma((X/A)_{\conv}, \calM) \iso M_S\otimes_{F_\infty} \Omega^\bullet_{F_\infty/A},$$ where $F_\infty\coloneqq E_{S,\infty}(J)$ and $M_S\in \mathrm{MIC}^\mathrm{fg}(S[\tfrac1p],J)$ is the module with flat connection associated to $\calM$ by Lemma \ref{IsocrystalsToMIC:l}.
\end{theo}
\begin{proof}Write $X=\Spec(C)$, so that we have ring maps $A\to S\to C$ with $A\to S$ smooth and $S\to C$ surjective with kernel $J$. As in Construction \ref{SmoothEmbedding:c}, choose a finitely generated polynomial $A$-algebra $P$, a surjection $\pi\colon P\twoheadrightarrow S$ with kernel $K$, and an $A$-algebra section $\sigma\colon S\to P^\wedge_K$ with $\pi\circ\sigma=\mr{id}_{S}$. Set $I=\Ker(P\twoheadrightarrow C)$; then $K\subseteq I$ and $\pi(I)=J$. Let $E_m^+$ denote $E_{P,m}^+(I)$, and write $M_m\coloneqq \calM(E^+_m)$, $\Omega^\bullet_m\coloneqq \Omega^\bullet_{E_m/A}$. By Theorem \ref{convDeRhamComparison1:t} applied to the embedding $X\hookrightarrow \Spec(P)$, we have
$$R\Gamma((X/A)_{\conv},\calM)\iso M\otimes_{E_\infty}\Omega^\bullet_{E_\infty/A},$$
and it remains to produce a canonical quasi-isomorphism between this and $M_S\otimes_{F_\infty}\Omega^\bullet_{F_\infty/A}$.

\spa

By Lemma \ref{ConvergentAffineSmooth:l}, the data $(\pi,\sigma)$ produces, for every $m\geq 0$, $A$-algebra homomorphisms
$$a_\infty\colon E_\infty\longrightarrow F_\infty,\qquad b_\infty\colon F_\infty\longrightarrow E_\infty,$$
satisfying $a_\infty\circ b_\infty=\mr{id}_{F_\infty}$. The isocrystal property of $\calM$ yields canonical identifications $M\otimes_{E_\infty}F_\infty=M_S$ and $M_S\otimes_{F_\infty,b_\infty}E_\infty=M$, so $a_\infty$ and $b_\infty$ induce morphisms of complexes
\begin{align*}
\Phi_\infty\colon\,& M\otimes_{E_\infty}\Omega^\bullet_{E_\infty/A}\longrightarrow M_S\otimes_{F_\infty}\Omega^\bullet_{F_\infty/A},\\
\Psi_\infty\colon\,& M_S\otimes_{F_\infty}\Omega^\bullet_{F_\infty/A}\longrightarrow M\otimes_{E_\infty}\Omega^\bullet_{E_\infty/A},
\end{align*}
with $\Phi_\infty\circ\Psi_\infty=\mr{id}$. It suffices to show that $\Psi_\infty\circ\Phi_\infty$, induced by the $A$-algebra endomorphism
$$\rho\coloneqq b_\infty\circ a_\infty\colon E_\infty\longrightarrow E_\infty,$$
is a quasi-isomorphism.

\spa

We proceed as in \stack{07LH}. Let $x_1,\dots,x_n$ be the polynomial generators of $P$ over $A$ and write $\rho(x_i)=x_i+z_i$ for some $z_i\in E_\infty$. Write $P'\coloneqq P[\xi_1,\dots, \xi_n]$, $I'\coloneqq (I,\xi_1,\dots,\xi_n)$, and $E'_\infty\coloneqq E_{P',\infty}(I')$. We have a decomposition of $\rho$ as $$E_\infty\xrightarrow{\lambda} E'_\infty\xrightarrow{\mu} E_\infty,$$ where $\lambda$ sends $x_i\mapsto x_i+\xi_i$ and $\mu$ sends $\xi_i\mapsto z_i$. If $\alpha$ is the automorphism of $E'_\infty$ defined by $\alpha(x_i)\coloneqq x_i-\xi_i$ and $\alpha(\xi_i)\coloneqq \xi_i$, then thanks to Lemma \ref{ConsequencePoincare:l}, we deduce that $\alpha\circ \lambda$ induces a quasi-isomorphism  $$M\otimes_{E_\infty}\Omega^\bullet_{E_\infty/A}\iso M\otimes_{E_\infty}\Omega^\bullet_{E'_\infty/A}.$$ Similarly, $\mu$ induces a quasi-isomorphism $$M\otimes_{E_\infty}\Omega^\bullet_{E'_\infty/A}\iso M\otimes_{E_\infty}\Omega^\bullet_{E_\infty/A}$$ by applying Lemma \ref{ConsequencePoincare:l} to the tautological inclusion $E_\infty\hookrightarrow E'_\infty$, which is a left inverse of $\mu$.
\end{proof}

Theorem \ref{convDeRhamComparison2:t} globalises thanks to the following vanishing result.
\begin{prop}\label{Technical:p}
	If $X$ is a closed subscheme of $\mathbb{A}^n_{A/p}$ and $\calM$ is a locally quasi-coherent sheaf of $\calO_\mr{conv}$-modules over $(X/A)_{\conv}$, then $$R\Gamma((X/A)_{\conv}, \calM\otimes_{\calO_\conv}\Omega^i_{\conv})$$ is homotopic to $0$ for all $i>0$.
\end{prop}
\begin{proof} Thanks to Lemma \ref{LimitModules:l} it is enough to prove that each $R\Gamma((X/A)_{\conv}^{(m)}, \calM\otimes_{\calO_\conv}\Omega^i_{\conv})$ is homotopic to $0$ for $i>0$ and $m\geq 0$. In addition, by Corollary \ref{CohomologyComparison:c}, we may replace  $(X/A)_{\conv}^{(m)}$ with $(X/A)_{\conv}^{(m),\mr{fg}}$.
	By Lemma \ref{DifferentialsTauEdgedCrystals:l}, the sheaves $\Omega^i_{\conv}$ are locally quasi-coherent in $(X/A)_{\conv}^{(m),\mr{fg}}$, which implies that $\calM\otimes_{\calO_\conv}\Omega^i_{\conv}$ is locally quasi-coherent. Thanks to Proposition \ref{InfFirstColumn:p} we can compute $R\Gamma((X/A)_{\conv}^{(m)}, \calM\otimes_{\calO_\conv}\Omega^i_{\conv})$ using the complex	$M_m(*)\otimes_{E_m(*)}\Omega^{i}_m(*),$ which is homotopic to zero by Lemma \ref{ConvAcyclicColumns:l}. This concludes the proof.
\end{proof}

We finally get the following extension of \cite[Thm. 0.6.6]{OgusCT}.

\begin{coro}\label{GlobalconvDeRhamComparison:c}
	Let $A$ be a Noetherian $p$-complete $p$-torsion free ring and $\mathfrak{Y}$ a smooth quasi-compact $p$-adic formal scheme over $A$. For every closed immersion $X\hookrightarrow \mathfrak{Y}\otimes_{\Zp}\Fp$ and every finitely generated convergent isocrystal $\calM$ of $X$ over $A$, there exists a canonical quasi-isomorphism
	$$R\Gamma((X/A)_{\conv}, \calM) \iso R\Gamma_\mr{dR}(]X[,\calM^\mr{ad}),$$ where $]X[$ is the open tube of $X$ in the generic fibre of $\mathfrak{Y}$ and $\calM^\mr{ad}$ is a flat connection over $]X[$ associated to $\calM$.
\end{coro}

\begin{proof}The connection $\calM^\mr{ad}$ is obtained by locally applying Construction \ref{IsocrystalConnection:c}, using the fact that the tube $]X[$ can be described as the increasing union of the adic generic fibres of the tubes $[X]^{(m)}$ defined in Construction \ref{Tubes:c}. To construct a morphism from the left hand side to the right hand side, we can then apply Proposition \ref{Technical:p} as in \stack{07LL}.  The morphism is a quasi-isomorphism by Theorem \ref{convDeRhamComparison2:t} applied locally on $\mathfrak{Y}$.
\end{proof}


\begin{thebibliography}{ABCD1}
	\bibitem[BdJ11]{BdJ11}
{ B.~Bhatt} and { A.~J.~de Jong}, \emph{Crystalline cohomology and de Rham cohomology}, \href{https://arxiv.org/abs/1110.5001}{arXiv:1110.5001} (2011).
\bibitem[Ber86]{Ber86}
{ P.~Berthelot}, Géométrie rigide et cohomologie des variétés algébriques de caractéristique $p$, in \emph{Introductions aux cohomologies $p$-adiques}, Mémoires de la Société Mathématique de France Nouvelle Série \textbf{23} (1986), 7--32.



\bibitem[D'A26]{EdgedCrystalline} { M. D'Addezio}, \textit{Edged crystalline cohomology}, in preparation, available at \href{https://daddezio.pages.math.cnrs.fr/edged-crys.pdf}{{https://daddezio.pages.math.cnrs.fr/edged-crys.pdf}} (2026).
\bibitem[D'A]{EdgedPrismatic} { M. D'Addezio}, \textit{Edged prismatic cohomology}, in preparation.

\bibitem[Dri22]{Dri22}
{ V. Drinfeld}, A stacky approach to crystals, in \emph{Dialogues between physics and mathematics: C. N. Yang at 100}, Springer, Cham, 2022, 19--47.
\bibitem[DvH22]{DvH22} M.~D'Addezio and P. van Hoften, \textit{Hecke orbits on Shimura varieties of Hodge type}, \href{https://arxiv.org/abs/2205.10344}{{arXiv:2205.10344}} (2022).
\bibitem[Ked20]{KedlayaCompanionsII}
{ K.~S.~Kedlaya}, \textit{\'Etale and crystalline companions, II},
\href{https://arxiv.org/abs/2008.13053}{{arXiv:2008.13053}} (2020).
\bibitem[KL15]{KL15}
{ K.~S.~Kedlaya} and { R.~Liu}, Relative $p$-adic Hodge theory: foundations, \emph{Astérisque} \textbf{371} (2015), 239 pp.
\bibitem[Ogu84]{Og84}
{ A. Ogus}, F-isocrystals and de Rham cohomology II---Convergent isocrystals, \emph{Duke Math. J.} \textbf{51} (1984), 765--850.
 \bibitem[Ogu90]{OgusCT} { A. Ogus},
{The Convergent Topos in Characteristic $p$},
in {\it The Grothendieck Festschrift, Vol. III}, Progress in Mathematics {\bf 88}, Birkhäuser, 1990, 133--162.
\bibitem[Sim96]{Sim96}
{ C.~Simpson}, Homotopy over the complex numbers and generalized de Rham cohomology, in \emph{Moduli of Vector Bundles} (M.~Maruyama, ed.), Dekker, 1996, 229--263.

\bibitem[SW13]{SW13}
{ P.~Scholze} and { J.~Weinstein}, Moduli of $p$-divisible groups, \emph{Camb. J. Math.} \textbf{1} (2013), 145--237.

\bibitem[Stacks]{Stacks} { The Stacks Project Authors}, \textit{Stacks Project}, available at \href{https://stacks.math.columbia.edu}{{https://stacks.math.columbia.edu}}.

\end{thebibliography}
\end{document}